\documentclass[journal,twoside]{IEEEtran}

\usepackage{cite}
\usepackage{graphicx}
\usepackage{amsmath}
\usepackage{amssymb}
\usepackage{dsfont}
\usepackage{amsthm}
\usepackage[utf8]{inputenc}
\newtheorem{thm}{Theorem}
\newtheorem{lemma}[thm]{Lemma}

\newtheorem{proposition}[thm]{Proposition}

\newcommand*{\Scale}[2][4]{\scalebox{#1}{$#2$}}
\renewcommand{\P}{\mathbb{P}_{\pi, P}}
\newcommand{\Po}{\mathbb{P}_{\pi^0,P^0}}

\newcommand{\E}{\mathbb{E}}

\newcommand{\K}[1]{\textup{KT}_k(#1)}
\newcommand{\KT}[2]{\mbox{KT}_{#1}(#2)}
\newcommand{\KThat}{\hat{k}_{\mbox{\tiny{KT}}}}
\renewcommand{\a}{\bold{a}_{n\times n}}
\newcommand{\z}{\bold{z}_n}

\newcommand{\zs}{\bold{z}_n^{\star}}
\newcommand{\pen}[1]{\mbox{pen}(#1)}
\newcommand{\pena}[2]{\mbox{pen}(#1,#2)}
\newcommand{\Z}{\bold{Z}_n}
\newcommand{\A}{\bold{A}_{n\times n}}

\usepackage{subcaption}

\DeclareMathOperator*{\argmax}{arg\,max}  

\usepackage{xcolor}

\begin{document}
\title{Estimation of the Number of Communities\\in the Stochastic Block Model }

\author{Andressa~Cerqueira and Florencia~Leonardi
\thanks{A. Cerqueira is with the Department
of Statistics, Universidade Federal de S\~ao Carlos, S\~ao Carlos,
SP, Brazil.}
\thanks{F. Leonardi is with Instituto de Matem\'atica e Estat\'\i stica, Universidade de S\~ao Paulo, S\~ao Paulo,
SP, Brazil.}
\thanks{Copyright (c) 2017 IEEE. Personal use of this material is permitted.  However, permission to use this material for any other purposes must be obtained from the IEEE by sending a request to 
 pubs-permissions@ieee.org}}


\maketitle

\begin{abstract}
In this paper we introduce an estimator for the number of communities in the Stochastic Block Model (SBM), 
based on the maximization of a penalized version of the so-called Krichevsky-Trofimov mixture distribution. 
We prove its eventual almost sure convergence to the underlying number of communities, 
without assuming a known upper bound on that quantity. Our results apply to both the dense and the sparse
regimes. To our knowledge this is the first consistency result for the estimation of the number of 
communities in the SBM in the unbounded case, that is when the number of communities is allowed to grow with the same size. 
\end{abstract}


%
\IEEEpeerreviewmaketitle

\section{Introduction}

In this paper we address  the model selection problem for the  Stochastic Block Model (SBM); that is, the estimation of the number of communities given a sample of the adjacency matrix.  The SBM was introduced by \cite{holland1983stochastic} and has rapidly popularized in the literature as a model for random networks
exhibiting blocks or communities between their nodes. In this model, each node in the network has associated a latent discrete random variable describing its community label, and given two nodes, the possibility of a  connection between them depends only on the values of the nodes' latent variables.  

From a statistical point of view, some methods have been proposed to  address the problem of parameter estimation or label recovering for the SBM. Some examples include maximum likelihood estimation \cite{bickel2009nonparametric, amini2013pseudo}, variational methods \cite{daudin2008mixture,latouche2012variational}, spectral clustering \cite{rohe2011spectral} and Bayesian inference \cite{van2017bayesian}.
The asymptotic properties of these estimators have also been considered in subsequent works such as \cite{bickel2013asymptotic} or \cite{su2017strong}. In \cite{abbe2017community} the reader can find an overview of recent approaches and theoretical results concerning the problem of community detection in SBMs. All these approaches assume the number of communities is known \emph{a priori}. 

The model selection problem, that is the 
estimation of the number of communities,  has been also addressed before using different approaches. 
Some examples include methods based on the spectrum of the graph \cite{le2015estimating,bickel2016hypothesis,lei2016goodness} or cross validation \cite{chen2018network,li2016network}. 
 From a Bayesian perspective, in \cite{daudin2008mixture} the authors propose a  criterion
 known as Integrated Completed Likelihood (ICL) based on the previous work \cite{biernacki2000} 
 for clustering, where a penalized profile likelihood function is used as approximation of the ICL. 
 To our knowledge it was not until \cite{wang2017likelihood} that a consistency result was obtained for a 
 model selection criterion. In the cited work the authors propose the maximization of the penalized 
 log-likelihood function  and show its convergence in probability  to the true number of communities.  
 Their proof only applies to the case where the number of candidate values for the estimator is finite
  (it is upper bounded by a known constant) and the network average degree grows at least as a 
  polylog function on the number of nodes. Moreover, the penalizing term is of order $n\log n$, with $n$
  the number of nodes in the network, a rate considerably bigger than the usual penalizing term  arising in the classical Bayesian Information Criterion \cite{schwarz1978}.
From a practical point of view, the computation of the log-likelihood function
 and its supremum is  not a simple task due to the hidden nature of the nodes' labels. However, 
 some approximate  versions of the estimator can be obtained  by variational methods
  using the EM algorithm  \cite{bickel2013asymptotic,daudin2008mixture},
a   profile maximum likelihood criterion as in \cite{bickel2009nonparametric} 
or the pseudo-likelihood algorithm in \cite{amini2013pseudo}.
 In a recent paper, the authors of  \cite{hu2019corrected} study a related method to the likelihood approach in  \cite{wang2017likelihood}, using the profiled conditional likelihood that they call corrected Bayesian Information Criterion. The hypothesis they  assume are the same as in \cite{wang2017likelihood}  and the penalty term is of order $n$.

In this paper we take an information-theoretic perspective and  introduce 
the Kri\-chevs\-ky-Tro\-fi\-mov (KT)  estimator  in order to determine the number of communities of a SBM  based on a sample of the adjacency matrix of the network. 
The KT estimator can be seen as a particular version of  the  Model Description Length (MDL) principle  \cite{barron1998} with KT code lengths \cite{kt1981} and 
has been previously proposed as a model selection criteria for  the memory of a Markov
 chain \cite{finesso1992estimation,csiszar2000consistency}, the context tree of a variable length Markov
  chain \cite{csiszar-talata-2006} or the number of hidden states in a Hidden Markov Model  \cite{liu1994order, gassiat2003optimal}. The proposed method is a penalized estimator based on 
  a mixture distribution of the model, known as Krichevsky-Trofimov mixture distribution. It can be seen as a Bayesian estimator with a particular choice for the 
  prior distributions, and it is somehow related to the approach proposed  in \cite{latouche2012variational}.

The main contribution of this work is the proof of the strong  consistency  
of the proposed estimator to select the number of communities in the SBM. 
 By  strong consistency  we mean that   eventually, the estimator equals the true number of communities with probability one, and the term should not be confused with the strong recovery notion in community detection problems \cite{zhao2012consistency}. 
 We prove the strong consistency  of the estimator in the \textit{dense} regime, where the probability of having an edge is considered to be constant, and in the \textit{sparse} regime where this probability goes to zero with $n$ having order $\rho_n$. The study of the second regime is more interesting in the sense that it is necessary to control how much information is required 
 to estimate the  parameters of the model. We prove the strong consistency in the sparse case  provided
  the expected degree of a given node grows to infinity, that is $n\rho_n \rightarrow \infty$, weakening the assumption in  \cite{wang2017likelihood}  that proves consistency in the regime $\frac{n\rho_n}{\log n} \to \infty$. We also consider a penalty function of smaller order compared to $n\log n$ used in  \cite{wang2017likelihood}   and we do not assume a known upper bound on the true number of communities.
To our knowledge, this is the first strong consistency result for an estimator of the number of communities, even in the bounded case, and the first one to prove consistency when the number of communities is allowed to grow with the sample size. 
We also investigate the performance of the variational approximation  introduced in \cite{latouche2012variational} 
and compare the performance of this algorithm with other methods on simulated networks. The simulation results show that the performance of the approximation to the KT estimator is comparable with other methods for balanced networks. However, this estimator performs better for unbalanced networks.  
 
The paper is organized as follows. In Section~\ref{defs} we define the model and the notation used in the paper, in Section~\ref{kt}  we introduce the KT estimator for the number of communities and state the main result. The proof of the consistency of the estimator is presented in Section~\ref{proof}.  In section \ref{simulations} we investigate the performance of the variation approximation of the estimator on simulated data. The final discussions are provided in Section \ref{discussion}.

\section{The Stochastic Block Model}\label{defs}

Consider a non-oriented random network with nodes $\{1,2,\dotsc, n\}$,  
specified by its adjacency  matrix $\A \in \{0,1\}^{n\times n}$ 
that is  symmetric and has  diagonal entries equal to zero. Each node 
$i$ has associated a latent (non-observed) variable $Z_i$ 
on $[k]:=\{1,2,\dotsc, k\}$, the \emph{community} label of node $i$.  

The SBM with $k$ communities is a probability model for a random 
network as above, where the latent variables $\Z=(Z_1,Z_2,\cdots,Z_n)$ 
are independent and identically distributed random variables over $[k]$ 
and the law of the adjacency matrix $\A$, conditioned on the value of 
the latent variables 
$\Z=\z$, is a product measure of Bernoulli random variables whose 
parameters depend only on the nodes' labels. More formally, there exists 
a probability distribution over $[k]$, denoted by $\pi=(\pi_1,\cdots,\pi_{k})
$, and a symmetric probability matrix  $P \in [0,1]^{k\times k}$ such that 
the distribution of the pair $(\Z,\A)$ is given by 

\begin{equation}\label{def-prob}
\P(\z,\a) = \prod_{1\leq a\leq k} \!\pi_{a}^{n_a}\!\! \prod_{1\leq a\leq b\leq k} \!\!P_{a,b}^{o_{a,b}} (1-P_{a,b})^{n_{a,b}-o_{a,b}}\,,
\end{equation}
where the counters $n_a=n_a(\z)$, $n_{a,b}=n_{a,b}(\z)$ and $o_{a,b}=o_{a,b}(\z,\a)$ are given by
\begin{align*}
n_a(\z) &= \sum\limits_{i=1}^n \mathds{1}\{z_i=a\}\, , \qquad\quad\;\, 1 \leq a \leq k\,,\\
n_{a,b}(\z) &=\begin{cases}
n_a(\z)n_b(\z)\, ,& 1 \leq a < b \leq k,\\
\frac12n_a(\z)(n_a(\z)-1)\, & 1 \leq a=b \leq k,\,
\end{cases}
\end{align*}
and  
\[
o_{a,b}(\z,\a) =  
\begin{cases} \sum\limits_{1\leq i, j\leq n} \mathds{1}\{z_i=a,z_j=b\}a_{ij} ,&a <  b,\\
\sum\limits_{1\leq i < j\leq n} \mathds{1}\{z_i=a,z_j=b\}a_{ij}  ,&a =b\,.
\end{cases}
\]
As it is usual in the definition of likelihood functions, by convention we define $0^0=1$ in \eqref{def-prob} when some of the parameters are 0. 

We denote by  
 $\Theta^k$ the  parametric space for a model with $k$ communities, given by
\begin{equation*}
\begin{split}
\Theta^k= \bigg\lbrace  (\pi,P)\colon  \pi \in (0,1]^k, \,\sum_{a=1}^k\pi_a=1,&\,P \in [0,1]^{k\times k},\\
&  P\,\text{ is symmetric}\bigg\rbrace\,.
\end{split}
\end{equation*}

The  \emph{order} of the SBM is defined as the smallest $k$ for which
the equality \eqref{def-prob} holds for a pair of parameters 
$(\pi^0, P^0)\in \Theta^k$ and will be denoted by $k_0$.  
If a SBM has order $k_0$ then  it cannot be reduced to a model with less communities than $k_0$; this specifically means that $P^0$ does not have two identical columns.

When $P^0$ is fixed and does not depend on $n$, the mean degree of a given node grows
 linearly in $n$ and this regime produces  very connected, dense graphs. 
In this paper we also consider the regime  producing sparse graphs 
 (with less edges), that occurs when $P^0$  decreases to the zero matrix   with $n$.  
 In this sparse regime  we write $P^0=\rho_nS^0$, where $S^0 \in [0,1]^{k\times k}$ does not 
 depend on $n$ and $\rho_n$ is a function decreasing to 0 at a sufficiently slow rate such that $n\rho_n \rightarrow \infty$.

\section{The KT order estimator}\label{kt}

Given a sample $(\z,\a)$ from the  distribution \eqref{def-prob} with parameters $(\pi^0,P^0)$,  where we assume we only observed the network $\a$,  the estimator of the number of communities  is  defined by
\begin{equation}\label{est_kt}
\hat{k}_{\mbox{\tiny{KT}}}(\a)=\argmax\limits_{1\leq k\leq n} \{ \,  \log \KT{k}{\a} - \pen{k,n}\,\}\,,
\end{equation}
where $\K{\a}$ is the integrated likelihood for a SBM with $k$ communities and $\pena{k}{n}$ is a penalizing function that will be specified later. The integrated likelihood $\K{\a}$ is obtained by integrating the likelihood of the model using a specific choice of prior distribution for the parameters $(\pi,P)$. In this specific setting, we choose as a prior distribution a product of a Dirichlet($1/2,\cdots, 1/2$), the prior distribution for $\pi$,  and a product of $(k^2+k)/2$ Beta($1/2,1/2$) distributions, the prior for the symmetric matrix $P$. Formally, we define the distribution $\nu_k(\pi, P)$ on  $\Theta^k$ as
\begin{equation}\label{mixture_dist}
\begin{split}
\nu_k(\pi, P)&= {\textstyle\frac{\Gamma\left(\frac{k}{2}\right)}{\Gamma\left(\frac{1}{2}\right)^k}}\prod_{1\leq a\leq k}\pi_{a}^{-\frac12} \\
& \qquad \times
 \;\prod_{1\leq a\leq b \leq k} {\textstyle\frac{1}{\Gamma\left(\frac{1}{2}\right)^2}}\;P_{a,b}^{-\frac12}(1-P_{a,b})^{-\frac12}
\end{split}
\end{equation}
and the integrated likelihood based on $\nu_k(\pi, P)$ is given by
\begin{equation}\label{kt-mix}
\begin{split}
\K{\a}& =\E_{\nu_k}[\,  \P(\a)  \,] \\
& = \int_{\Theta^k}\P(\a)\nu_k(\pi, P) d\pi dP \,,
\end{split}
\end{equation}
where $\P(\a)$ stands for the marginal distribution obtained from \eqref{def-prob}, that is 
\begin{equation}\label{eq:likelihood_function}
 \P(\a)=\sum\limits_{\z \in [k]^n} \P(\z,\a)\,.
\end{equation}

The distribution given in \eqref{kt-mix} is the integrated marginal  likelihood of the model, also known as model evidence under a Bayesian perspective, see  for example the related work \cite{latouche2012variational}. Because of the specifc choice of $\nu_k(\pi, P)$, in this paper we will follow the information-theoretical tradition and call the integrated likelihood given in \eqref{kt-mix} the Krichevsky-Trofimov mixture and the derived estimator for the number of communities \eqref{est_kt} the KT estimator. 

As in other model selection problems  where the KT approach has proved to be very useful, see for example  \cite{liu1994order, gassiat2003optimal,csiszar-talata-2006}, 
 in the case of the SBM there is a closed relationship between the KT mixture distribution and the maximum likelihood function. 
The following proposition shows non asymptotic uniform bounds for the log-likelihood function in terms of the logarithm of the KT distribution. Its proof
is postponed to the Appendix.

\begin{proposition}\label{prop:razao_L_KT}
For all $k$, all $n\geq \max(4,k)$ and all $\a$ we have that 
\begin{align}
\log\K{\a} \;& \leq \; \log\sup_{(\pi,P) \in \Theta^k} \P(\a) \\
& \;\leq\; \log\K{\a} + \textstyle\frac{k(k+2)-1}{2} \log n + c_k\notag
\end{align}
where 
\begin{equation}\label{def:ck}
c_{k} = k(k+1) + 1\,. 
\end{equation}
\end{proposition}

Proposition~\ref{prop:razao_L_KT}  is at the core of the proof of the 
consistency  of $\hat{k}_{\mbox{\tiny{KT}}}$ defined by \eqref{est_kt}. 
In order to derive the strong consistency result for the  KT order estimator, 
we need a penalty 
function in  \eqref{est_kt}  with a given rate of convergence when $n$ grows to infinity. 
Although there is a range of possibilities for this penalty function, 
the specific form we use in this paper    is 
\begin{equation}\label{eq:penalty}
\begin{split}
\pena{k}{n} 
& =\Bigl[ \textstyle\frac{k(k-1)(2k-1)}{12} + \textstyle\frac{k(k-1)}{2} + \textstyle\frac{(1+\epsilon)(k-1)}{2}  \Bigr] \log n
\end{split}
\end{equation}
for any $\epsilon>0$.
The convenience of the expression above will be make clear in the proof of the consistency result.  
 Observe that the penalty function defined by \eqref{eq:penalty} is dominated by a term of order  $k^3\log n$ and then  it is of smaller order than the function $\frac{k(k+1)}{2} n\log n$ used in \cite{wang2017likelihood}. For a model selection criterion, a too strong penalty term can lead to a bigger
 probability of  underestimating the true number of communities, then a small penalty term is in general 
 desirable. 

We finish this section by stating  the main theoretical result in this paper. \\
\begin{thm}\label{the:estimators_convergence}
Suppose the SBM has order $k_0$ with parameters $(\pi^0,P^0)$, and suppose 
$\pen{k,n}$ is given by  \eqref{eq:penalty}. 
Then we have that
\[
\hat{k}_{\mbox{\tiny{KT}}}(\a) = k_0 
\]
eventually almost surely as $n\to\infty$. 
\end{thm}

The proof of this and other auxiliary results are given in the next section and in the Appendix. 

\section{Proof of the Consistency Theorem}\label{proof}

The proof of Theorem~\ref{the:estimators_convergence} is divided in two main parts. The first one, presented in Subsection~\ref{non-over}, proves that  $\hat{k}_{\mbox{\tiny{KT}}}(\a)$  does not overestimate the true order $k_0$, eventually almost surely when $n\to\infty$, even without assuming a known upper bound on $k_0$. The second part of the proof, presented in Subsection~\ref{non-under}, shows that $\hat{k}_{\mbox{\tiny{KT}}}(\a)$  does not underestimate $k_0$, eventually almost surely when $n\to \infty$. By combining these two results we prove that $\hat{k}_{\mbox{\tiny{KT}}}(\a) = k_0$ eventually almost surely as $n\to\infty$. 

\subsection{Non-overestimation}\label{non-over}

The main result in this subsection is given by the following proposition. \\

\begin{proposition}\label{prop:no_overestimation} 
Let $\a$ be a sample of size $n$ from a SBM of order $k_0$, with parameters $\pi^0$ and $P^0$. 
Then, the $\hat{k}_{\mbox{\tiny{KT}}}(\a)$ order estimator defined in \eqref{est_kt}  does not 
overestimate $k_0$, eventually almost surely when $n\to\infty$. 
\end{proposition}

The proof of Proposition~\ref{prop:no_overestimation} follows straightforward 
from Lemmas~\ref{lemma:k0_log} and  \ref{lemma:log_n} presented below.
These lemmas are inspired in the work \cite{gassiat2003optimal} which proves consistency for an order estimator of a  Hidden Markov Model (HMM). 
 In any case, we would like to emphasise that even if the SBM can be seen as a  
``hidden variable model'', there are substantial differences with HMM, the most important one being that  in the case of a SBM,  when a new node is added there are $n$ possible new edges in the network, depending on the labels of all previous nodes. In contrast,  in a HMM the observable only depends on the state at time $n$.



\begin{lemma}\label{lemma:k0_log}
Under the hypotheses of Proposition~\ref{prop:no_overestimation}  we have that
$$\hat{k}_{\mbox{\tiny{KT}}}(\a) \not\in (k_0,\log n]$$
eventually almost surely when $n\to\infty$.
\end{lemma}

\begin{proof}
First observe that 
\begin{equation}\label{firsteq}
\begin{split}
&\Po( \hat{k}_{\mbox{\tiny{KT}}}(\a) \in (k_0, \log n] ) \\
&\qquad\qquad\;=\;  \sum\limits_{k=k_0+1}^{\log n} \Po( \hat{k}_{\mbox{\tiny{KT}}}(\a) = k )\,.
\end{split}
\end{equation}
Using Lemma~\ref{prop:ineq_prob_k} we can bound the sum in the right-hand side by 
\begin{align}\label{bound1}
\sum\limits_{k=k_0+1}^{\log n} &\exp\left\lbrace   \textstyle\frac{k_0(k_0+2)-1}{2}\log n + c_{k_0} + d_{k_0,k,n}\right\rbrace\\
 &\leq  e^{ c_{k_0}}\,\log n \,\exp\left\lbrace   \textstyle\frac{k_0(k_0+2)-1}{2}\log n + d_{k_0,k_0+1,n} \right\rbrace,\notag
\end{align}
where $d_{k_0,k,n} = \pen{k_0,n}-\pen{k,n}$ and $c_{k_0}=k_0(k_0+1)+1$.  
A simple calculation gives that
\begin{equation*}
\begin{split}
\pena{k}{n} 
&= \sum\limits_{i=1}^{k-1}\bigl[\textstyle\frac{(i(i+2)+1+\epsilon}2\bigr]\log n  \\
\end{split}
\end{equation*}
and therefore 
\begin{equation*}
\begin{split}
\textstyle\frac{k_0(k_0+2)-1}{2}&\,\log n + d_{k_0,k_0+1,n}\\[1mm]
&\;=\textstyle\Bigl(\frac{k_0(k_0+2)-1}{2} 
- \textstyle{\frac{ k_0(k_0+2)+1+\epsilon}{2}}\Bigr) \log n\\
&\;= -(1+\epsilon/2)\log n\,.\\
\end{split}
\end{equation*}
By using this expression in the right-hand side of \eqref{bound1} to bound \eqref{firsteq}
we obtain 
that 
\begin{align*}
\sum_{n=1}^{\infty}\,\textstyle\Po( \hat{k}_{\mbox{\tiny{KT}}}(\a) \in (k_0, \log n] ) &\;\leq\; e^{c_{k_0}} \sum\limits_{n=1}^{\infty} \frac{\log n}{n^{1+\epsilon/2}} \\
& \;<\; \infty\;.
\end{align*}
Now the result follows by the first Borel Cantelli lemma. 
\end{proof}



\begin{lemma}\label{lemma:log_n}
Under the hypotheses of Proposition~\ref{prop:no_overestimation}  we have that
$$\hat{k}_{\mbox{\tiny{KT}}}(\a) \not\in (\log n, n]$$
eventually almost surely when $n\to\infty$.
\end{lemma}

\begin{proof}
As in the proof of Lemma~\ref{lemma:k0_log} we write
\begin{equation}
\begin{split}
&\Po(\hat{k}_{\mbox{\tiny{KT}}}(\a) \in ( \log n , n]\, ) \\
&\qquad\qquad=  \sum\limits_{k=\log n}^{n} \Po( \hat{k}_{\mbox{\tiny{KT}}}(\a) = k )
\end{split}
\end{equation}
and we use again Lemma~\ref{prop:ineq_prob_k} to bound the sum in the right-hand side by 
\begin{align}\label{main-ineq}
&\sum\limits_{k=\log n}^{ n} \exp\left\lbrace   \textstyle\frac{k_0(k_0+2)-1}{2}\log n + c_{k_0} + d_{k_0,k,n} \right\rbrace\notag\\
&\;\; \leq\; e^{ c_{k_0}}\, n\, \exp \Bigl\{ \textstyle \log n \,\Bigl[ \frac{k_0(k_0+2)-1}{2} +
\frac{d_{k_0,\log n,n}}{\log n} \Bigr] \Bigr\} \,.
\end{align}
Since $\pen{k,n}{}/\log(n)$ does not depend on $n$ and increases cubically in $k$ we have that
\begin{align*}
{\textstyle\frac{k_0(k_0+2)-1}{2}} +\textstyle{ \frac{d_{k_0,\log n,n}}{\log n}} \;&=\;
{\textstyle\frac{k_0(k_0+2)-1}{2}} + \textstyle{\frac{\text{pen}(k_0,n)}{\log n}} \\
& \quad- {\textstyle\frac{\text{pen}(\log n,n)}{\log n}} \\
&<\;  -3
\end{align*}
for all sufficiently large $n$. Thus 
 summing  \eqref{main-ineq}  in $n$ we obtain
\begin{equation*}
\sum\limits_{n=1}^{\infty}\,\textstyle e^{ c_{k_0}} \,n\,\exp\Bigl\lbrace \log n \,\Bigl[  
 {\textstyle \frac{k_0(k_0+2)-1}{2} }+  d_{k_0,\log n, n} \Bigr]\Bigr\rbrace \;<\; \infty
\end{equation*}
and the result follows from the first Borel Cantelli lemma.
\end{proof}

\subsection{Non-underestimation}\label{non-under}
\renewcommand{\theta}{\pi,P}

In this subsection we deal with the proof of the non-underestimation of $\KThat(\a)$. The main result of this section is the following

\begin{proposition}\label{prop:no_underestimation}
Let $\a$ be a sample of size $n$ from a SBM of order $k_0$ with parameters $(\pi^0,P^0)$.
 Then, the $\hat{k}_{\mbox{\tiny{KT}}}(\a)$ order estimator defined in \eqref{est_kt} 
  does not underestimate $k_0$, eventually almost surely when $n\to\infty$. 
\end{proposition}

In order to prove this result we need  Lemmas~\ref{lemma:ratio_underfitting_rho} and \ref{lemma:ratio_underfitting2}
below, that explore limiting properties of the under-fitted model. That is we handle with the problem of fitting  a SBM of order $k_0$ in the parameter space $\Theta^{k_0-1}$.

An intuitive construction of a ($k-1$)-block model  from a $k$-block model is obtained by merging two given blocks. This merging can be implemented in several ways, but  here we consider the construction given in \cite{wang2017likelihood}, with the difference that instead of using the sample block proportions we use the limiting  distribution $\pi$ of the original $k$-block model. 

Given  $(\pi,P)\in \Theta^{k}$ 
we define the merging operation $M_{a,b}(\pi,P) = (\pi^*,P^*)\in \Theta^{k-1}$ which combines blocks with labels $a$ and $b$. For ease of exposition we only  show the explicit  definition for the case $a=k-1$ and $b=k$. 
In this case, the merged distribution $\pi^*$ is given by 
\begin{align}\label{eq:merging1}
\pi^*_i &= \pi_i\,  \hspace{2cm} \text{for } 1 \leq i \leq k-2\,,\\
\pi^*_{k-1} &= \pi_{k-1}+\pi_{k}\,.\notag
\end{align}
On the other hand, the merged matrix $P^*$
is obtained as 
\begin{align}\label{eq:merging}
&P^*_{l,r} = P_{l,r} \hspace{4cm} \text{for } 1 \leq l,r \leq k-2\,,\notag\\[2mm]
&P^*_{l,k-1} = \frac{\pi_l\pi_{k-1}P_{l,k-1}+ \pi_l\pi_{k}P_{l,k}}{\pi_l\pi_{k-1}+ \pi_l\pi_{k}} \hspace{0.6cm} \text{for } 1 \leq l \leq k-2\,,\\[2mm]
&P^*_{k-1,k-1} = \frac{\pi_{k-1}^2P_{k-1,k-1}+ 2\pi_{k-1}\pi_{k}P_{k-1,k} + \pi_{k}^2P_{k,k}}{\pi_{k-1}^2+ 2\pi_{k-1}\pi_{k}+ \pi_{k}^2} \,.\notag
\end{align}
For arbitrary $a$ and $b$ the definition is obtained by permuting the labels. 

Given $\a$  originated from the SBM of order $k_0$ and parameters $(\pi^0,P^0)$,  
we define the profile likelihood estimator of the label assignment under the ($k_0-1$)-block model as
\begin{equation}\label{profileest}
\zs \; =\; \argmax \limits_{\z \in [k_0-1]^n}\;\sup\limits_{(\theta) \in \Theta^{k_0-1}} \P(\z, \a)\,.
\end{equation}

The next lemmas show that the  logarithm of the ratio between the maximum  likelihood under the true order $k_0$  and the maximum profile likelihood under the under-fitting $k_0-1$ order model is bounded from below by a function growing faster than $n\log n$, eventually almost surely when $n\to\infty$. Each lemma consider one of the two possible regimes $\rho_n=\rho>0$ (dense regime) or  $\rho_n \rightarrow 0$ at a rate $n\rho_n \rightarrow \infty$ (sparse regime).

\begin{lemma}[dense regime]\label{lemma:ratio_underfitting_rho} Let $(\z,\a)$ be a sample of size $n$ from a SBM of order $k_0$ with parameters $(\pi^0,P^0)$,  with $P^0$ not depending on $n$. Then there exist $r,s \in [k_0]$ such that for $(\pi^*,P^*) = M_{r,s}(\pi^0,P^0)$ we have that almost surely
\begin{align}\label{eq:lim_ratio_underfitting_rho}
&\liminf\limits_{n\rightarrow \infty}\;\dfrac{1}{n^2}\log\dfrac{\sup_{(\theta) \in \Theta^{k_0}}\P(\z,\a)}{\sup_{(\theta) \in \Theta^{k_0-1}}\P(\zs,\a)} \notag\\
&\quad \geq\; { \frac 12}\Biggl[ \sum_{ 1\leq a,b\leq k_0}\pi^0_a\pi^0_b\,\gamma(P^0_{ab}) - \!\sum_{ 1\leq a,b\leq k_0-1}\pi^*_a \pi^*_b\,\gamma(P^*_{a,b})\Biggr]\notag\\
&\quad >\;0\,,
\end{align}
where $\gamma (x)=x\log x + (1-x)\log (1-x)$.
\end{lemma}

\begin{proof}
Given $k$ and ${\bf\bar z}_n \in [k]^n$ define the empirical probabilities
\begin{equation}\label{empprob}
\begin{split}
\hat{\pi}_a({\bf\bar z}_n) &= \dfrac{n_a({\bf\bar z}_n)}{n}\, , \hspace{2cm} 1\leq a \leq  k\\
 \hat{P}_{a,b}({\bf\bar z}_n,\a) &= \dfrac{o_{a,b}({\bf\bar z}_n,\a)}{n_{a,b}({\bf\bar z}_n)}\, , \hspace{0.6cm} 1\leq a\leq b \leq  k\,.
\end{split}
\end{equation}
In the case $k=k_0$ the maximum log-likelihood function is given by  
\begin{equation*}\label{eq:complete_z0}
\begin{split}
 \log&\sup_{(\theta) \in \Theta^{k_0}}\P(\z,\a) = n \sum\limits_{1\leq a\leq k_0}{\hat{\pi}_a(\z)} \log \hat{\pi}_a(\z)\\
&\qquad\qquad+\sum\limits_{1\leq a\leq b\leq k_0} n_{a,b}(\z)\gamma( \hat{P}_{a,b}(\z,\a))\,.
\end{split}
\end{equation*}
Using that $n_{a,b}(\z)=n_a(\z)n_b(\z)$ for $a\neq b$ and ${ n_{a,a}(\z)=\frac12n_a(\z)(n_a(\z)-1)}$, 
dividing by $n^2$ and using  the Strong Law of Large Numbers 
we have that almost surely 
\begin{equation}\label{eq:lim_zero}
\begin{split}
\lim\limits_{n \rightarrow \infty}&\;\frac{1}{n^2}\log\sup_{(\pi,P) \in \Theta^{k_0}}\P(\z,\a) \\
&\qquad\qquad\;=\; { \frac 12} \sum\limits_{ 1\leq a,b\leq k_0}\pi^0_a\pi^0_b\,\gamma(P^0_{a,b})\, .
\end{split}
\end{equation}
Similarly for $k_0-1$ and $\zs \in [k_0-1]^n$  we have that almost surely 
\begin{equation}\label{eq:lim_tilde}
\begin{split}
&\limsup\limits_{n \rightarrow \infty}\;\dfrac{1}{n^2}\,\log \sup\limits_{(\theta) \in \Theta^{k_0-1}}\P(\zs, \a)\\
& \qquad\qquad=  { \frac 12} \sum\limits_{ 1\leq a,b\leq k_0-1}\tilde\pi_a\tilde\pi_b\,\gamma(\tilde P_{a,b})\,,
\end{split}
\end{equation}
for some $(\tilde\pi,\tilde P) \in \Theta^{k_0-1}$.
Combining \eqref{eq:lim_zero} and \eqref{eq:lim_tilde} we have that almost surely 
\begin{equation}\label{eq:liminf_eq}
\begin{split}
 &\liminf\limits_{n\rightarrow \infty}\;\dfrac{1}{n^2}\log\dfrac{\sup_{(\theta) \in \Theta^{k_0}}\P(\z,\a)}{\sup_{(\theta) \in \Theta^{k_0-1}}\P(\zs,\a)}\\[2mm]
&=\;  { \frac 12}\sum\limits_{ 1\leq a,b\leq k_0} \pi^0_a\pi^0_b\,\gamma(P^0_{a,b}) -  {\frac 12} \sum\limits_{1\leq a,b\leq k_0-1}\tilde\pi_a \tilde\pi_b\, \gamma( \tilde P_{a,b})\,.
\end{split}
\end{equation}
To obtain a lower bound for  \eqref{eq:liminf_eq} we need to compute $(\tilde\pi, \tilde P)$ that minimizes the right-hand side. 
This is equivalent to obtain $(\tilde\pi, \tilde P) \in \Theta^{k_0-1}$ that maximizes the second term  
\begin{equation}\label{eq:function_tilde}
\sum\limits_{1\leq a,b\leq k_0-1}\tilde\pi_a \tilde\pi_b\, \gamma( \tilde P_{a,b})\,.
\end{equation}
Denote by $(\bold{\widetilde{Z}}_n,\bold{\widetilde{X}}_{n\times n})$ a $(k_0-1)$-order SBM with distribution  $(\tilde\pi,\tilde P)$.
By definition
\[
\tilde P_{\tilde a,\tilde b} = \frac{P(\tilde X_{i,j}=1,\tilde Z_i=\tilde a,\tilde Z_j=\tilde b)}{P(\tilde Z_i=\tilde a,\tilde Z_j=\tilde b)}\,.
\]
Observe that when $\bold{\widetilde{X}}_{n\times n}=\A$,  the numerator equals
\begin{align*}
&\sum\limits_{a=1}^{k_0}\sum\limits_{b=1}^{k_0} \Scale[0.85]{ P(X_{i,j}=1|Z_i=a,Z_j= b)P(Z_i=a,Z_j= b,\tilde Z_i=\tilde a,\tilde Z_j= \tilde b)}\\
& =\sum_{a=1}^{k_0}\sum_{b=1}^{k_0}\Scale[0.85]{P(Z_i=a,\tilde Z_i=\tilde a)\,P^0_{a,b}\, P(Z_j= b,\tilde Z_j= \tilde b)} \\[2mm]
&= (QP^0Q^T)_{\tilde{a},\tilde{b}}\,,
\end{align*}
where $Q(\tilde a,a)$ denotes  a joint distribution on $[k_0-1]\times [k_0]$ (a coupling) with marginals  $\tilde\pi$ and $\pi^0$, respectively. 
Similarly, the denominator can be written as 
\begin{align*}
\sum_{a=1}^{k_0}\sum_{b=1}^{k_0} \Scale[0.85]{P(Z_i=a,\tilde Z_i=\tilde a)P(Z_j= b,\tilde Z_j= \tilde b)}=(Q(\bold{1}\bold{1}^T)Q^T)_{\tilde{a},\tilde{b}}\,,
\end{align*}
where $\bold{1}\bold{1}^T$ denotes the matrix with dimension $k_0\times k_0$ and all entries equal to 1.

 Then we can rewrite \eqref{eq:function_tilde} as
\begin{equation}\label{eq:function_tilde_matrix}
\sum\limits_{ 1\leq a, b\leq k_0-1}  (Q(\bold{1}\bold{1}^T)Q^T)_{a, b}\,\gamma 
\left[  \dfrac{(QP^0Q^T)_{a,b}}{ (Q(\bold{1}\bold{1}^T)Q^T)_{a,b}}  \right]\,.
\end{equation}
Therefore, finding  a pair $(\tilde \pi, \tilde P)$ maximizing \eqref{eq:function_tilde} is equivalent to finding an optimal coupling $Q$ maximizing \eqref{eq:function_tilde_matrix}.  In \cite{wang2017likelihood} the authors proved that there exist
 $r,s \in [k_0]$ such that \eqref{eq:function_tilde_matrix} achieves its maximum 
 at $(\pi^*,P^*) = M_{r,s}(\pi^0,P^0)$, see Lemma~A.2 there. This concludes the proof of the first inequality in \eqref{eq:lim_ratio_underfitting_rho}. In order to prove the second strict inequality in \eqref{eq:lim_ratio_underfitting_rho}, we consider for convenience and without loss of generality,  $r=k_0-1$ and $s=k_0$ (the other cases can be handled by a permutation of the labels).
Notice that in the right-hand side of  \eqref{eq:liminf_eq}, with $(\tilde\pi,\tilde P)$
substituted by the optimal value $M_{k_0-1,k_0}(\pi^0,P^0)$ defined by 
\eqref{eq:merging1} and \eqref{eq:merging}, all the terms with $1\leq a,b\leq k_0-2$
cancel. Moreover, as $\gamma$ is a convex function, Jensen's inequality implies 
that
\begin{equation}\label{jensen1}
\begin{split}
\pi^*_a\pi^*_{k_0-1}&\gamma(P^*_{a,k_0-1}) \;\leq\;\\
 &\pi^0_a\pi^0_{k_0-1}\gamma(P^0_{a,k_0-1})+ \pi^0_a\pi^0_{k_0}\gamma(P^0_{a,k_0})
 \end{split}
\end{equation}
for all $a=1,\dotsc, k_0-2$ and similarly 
\begin{equation}\label{jensen2}
(\pi^*_{k_0-1})^2\gamma(P^*_{k_0-1,k_0-1}) \;\leq\; \sum_{ k_0-1\leq a,b\leq k_0} \pi^0_a \pi^0_b\gamma(P^0_{a,b})\,.
\end{equation}
The equality holds for all $a$  in \eqref{jensen1} and in   \eqref{jensen2}  simultaneously  if and only if 
\[
P^0_{a,k_0} \;=\; P_{a,k_0-1} \qquad\text{ for all }a=1,\dotsc, k_0\,,
\]
in which case the matrix $P^0$ would have two identical columns, contradicting the fact that the sample $(\z,\a)$ originated from a SBM with order $k_0$.  Therefore the strict inequality must hold  in \eqref{jensen1} for at least one $a$ or in \eqref{jensen2}, showing that the second inequality in \eqref{eq:lim_ratio_underfitting_rho} holds. 
\end{proof} 

\begin{lemma}[sparse regime]\label{lemma:ratio_underfitting2} 
Let $(\z,\a)$ be a sample of size $n$ from a SBM of order $k_0$ with parameters 
$(\pi^0,\rho_n S^0)$, where $\rho_n \rightarrow 0$ at a rate $n\rho_n \rightarrow \infty $. Then there exist $r,s \in [k_0]$ such that for $(\pi^*,P^*) = M_{r,s}(\pi^0,S^0)$ we have that almost surely 
\begin{align}\label{eq:lim_ratio_underfitting2}
&\liminf\limits_{n\rightarrow \infty} \;\dfrac{1}{\rho_nn^2}\log\dfrac{\sup_{(\theta) \in \Theta^{k_0}}\P(\z,\a)}{\sup_{(\theta) \in \Theta^{k_0-1}}\P(\zs,\a)} \notag\\
&\qquad \geq\;\;{ \frac 12}\Biggl[ \sum_{a,b=1}^{k_0}\pi^0_a\pi^0_b\,\tau(S^0_{a,b}) - \sum_{a,b=1}^{k_0-1}\pi^*_a \pi^*_b\,\tau(P^*_{a,b})\Biggr]\\
&\qquad >\;0\,,\notag
\end{align}
where $\tau(x)=x\log x - x$.
\end{lemma}
\begin{proof}
This proof follows the same arguments used in the proof of 
Lemma~\ref{lemma:ratio_underfitting_rho},  but as in this case $P^0$ decreases to 
0 some limits must be handled differently.  
As before we have for $k_0$ and $\z$ that 

\begin{equation}\label{baseq1}
\begin{split}
 \log\sup_{(\theta) \in \Theta^{k_0}}&\P(\z,\a) = 
\,n \sum\limits_{1\leq a\leq k_0}{\hat{\pi}_a(\z)} \log \,\hat{\pi}_a(\z) \\
&+ \sum\limits_{1\leq a\leq b\leq k_0} n_{a,b}(\z)\gamma( \hat{P}_{a,b}(\z,\a) )\,.
\end{split}
\end{equation}
For $\rho_n \rightarrow 0$ we have that 
\[
\gamma(\rho_n x) = \rho_n(x\log x - x)+ x\rho_n\log \rho_n+ O(\rho_n^2 x^2)\,,
\]
see \cite[Supplementary material]{bickel2009nonparametric}. Therefore 
\begin{align}\label{eq:bickel_2009}
& \sum\limits_{ 1\leq a\leq b\leq k_0} n_{a,b}(\z)\gamma( \hat{P}_{a,b}(\z,\a) ) \notag \\
& \quad = \rho_n \sum\limits_{ 1\leq a\leq  b\leq k_0} n_{a,b}(\z)\,\tau\Bigl(  \frac{\hat{P}_{a,b}(\z,\a)}{\rho_n} \Bigr)\notag \\ 
& \qquad + E_n\log \rho_n + O\Bigl( \hat{P}_{a,b}(\z,\a)^2\Bigr)\,,
\end{align}
where $E_n=\sum\limits_{ 1\leq a\leq b\leq k_0} o_{ab}(\z,\a) $ (the total number of edges in the graph) and $\tau(x)=x\log x - x$. A similar expression can be obtained for 
$k_0-1$ and $\zs$. Thus, dividing by $\rho_nn^2$ and taking limits 
we have that
there must exists some $(\tilde \pi, \tilde S)\in \Theta^{k_0-1}$ such that almost surely 
\begin{equation}\label{eq:liminf_eq2}
\begin{split}
&\liminf\limits_{n\rightarrow \infty}\dfrac{1}{\rho_nn^2}\log\dfrac{\sup_{(\theta) \in \Theta^{k_0}}\P(\z,\a)}{\sup_{(\theta) \in \Theta^{k_0-1}}\P(\zs,\a)}\\
&\; = { \frac 12}\Biggl[\sum\limits_{1\leq a, b\leq k_0} \pi^0_a\pi^0_b\,\tau(S^0_{a,b}) - \!\!
 \sum\limits_{1\leq a, b\leq k_0-1}\tilde{\pi}_a\tilde{\pi}_b\, \tau( \tilde{S}_{a,b})\Biggr]\,.
\end{split}
\end{equation}
As before, we want to obtain $(\tilde{\pi}, \tilde{S})\in \Theta^{k_0-1}$ that maximizes the second term in the right-hand side of the equality above. 
The rest of the proof here is analogous to that of  Lemma~\ref{lemma:ratio_underfitting_rho}, 
by observing that $\tau$ is also a convex function and therefore the difference in \eqref{eq:lim_ratio_underfitting2}  is lower bounded by 0. 
%
%
\end{proof}

\begin{proof}[Proof of Proposition~\ref{prop:no_underestimation}]
To prove that $\hat{k}_{\mbox{\tiny{KT}}}(\a)$ does not underestimate $k_0$ it is enough to show that for all $k < k_0$
\[
\log \KT{k_0}{\a} - \pena{k_0}{n} \; >\;  \log \KT{k}{\a} - \pena{k}{n}
\]
eventually almost surely when $n\to\infty$.
As 
\[
 \lim_{n \rightarrow \infty}\;\dfrac{1}{\rho_n n^2}  \Bigl[\pena{k_0}{n} - \pena{k}{n}
  \Bigr] \;=\; 0
\]
this is equivalent to show that
\begin{equation*}\label{eq:eq_lim_kt}
 \liminf_{n \rightarrow \infty}\;\dfrac{1}{\rho_n n^2}\, \log \dfrac{\KT{k_0}{\a}}{\KT{k}{\a} }  \;>\; 0\,.
\end{equation*}
First note that the logarithm above can be  written as 
\begin{align*}
\log \,\dfrac{\KT{k_0}{\a}}{\KT{k}{\a} }\;= \;&
\log \,\dfrac{\KT{k_0}{\a}}{\sup_{(\theta) \in \Theta^{k_0}}\P(\a)} \\[2mm]
&+ \log \,\dfrac{\sup_{(\theta) \in \Theta^{k_0}}\P(\a)}{\KT{k}{\a}}  \,.
\end{align*}
Using Proposition~\ref{prop:razao_L_KT} we have that the first term in the right-hand side can be bounded below by
\begin{equation}\label{proof:overestimation_ineq1}
\Scale[0.9]{\log \,\dfrac{\KT{k_0}{\a}}{{\sup_{(\theta) \in \Theta^{k_0}}\P(\a)}} \;\geq\; - \textstyle\frac{k_0(k_0+2)-1}{2} \log n - c_{k_0}\,.}
\end{equation}
On the other hand, the second term can be lower-bounded by using  
\begin{align}\label{proof:overestimation_ineq2}
&\Scale[0.9]{\log \,\dfrac{\sup_{(\theta) \in \Theta^{k_0}}\P(\a)}{\KT{k}{\a}}}\notag\\
&\Scale[0.9]{= \;\log \dfrac{\sup_{(\theta) \in \Theta^{k_0}}\P(\a)}{{\sup_{(\theta) \in \Theta^{k}}\P(\a)}} \;+\; \log \dfrac{{\sup_{(\theta) \in \Theta^{k}}\P(\a)}}{\KT{k}{\a}}}\notag\\
&\Scale[0.9]{\geq\; \log\dfrac{\sup_{(\theta) \in \Theta^{k_0}}\P(\a)}{{\sup_{(\theta) \in \Theta^{k}}\P(\a)}}} \,.
\end{align}
By combining \eqref{proof:overestimation_ineq1} and \eqref{proof:overestimation_ineq2} we obtain 
\begin{align*}
\dfrac{1}{\rho_n n^2}\log &\dfrac{\KT{k_0}{\a}}{\KT{k}{\a} } \;\geq\; - \textstyle\frac{k_0(k_0+2)-1}{2} \dfrac{\log n}{\rho_n n^2} - \dfrac{c_{k_0}}{\rho_n n^2}\\
&\qquad\qquad\quad +\dfrac{1}{\rho_n n^2}\log\dfrac{\sup_{(\theta) \in \Theta^{k_0}}\P(\a)}{{\sup_{(\theta) \in \Theta^{k}}\P(\a)}}\,.
\end{align*}
Now, as  $n \rho_n\rightarrow \infty$ it suffices to show that for $k < k_0$, almost surely we have 
\begin{equation}\label{eq:lim_razao_proban}
\liminf\limits_{n \rightarrow \infty}\,\dfrac{1}{\rho_n n^2}\log\dfrac{\sup_{(\theta) \in \Theta^{k_0}}\P(\a)}{{\sup_{(\theta) \in \Theta^{k}}\P(\a)}} \; >\; 0\,.
\end{equation}
We start with $k=k_0-1$. Using $\zs$ defined by \eqref{profileest}
we have that 
\begin{equation*}
\begin{split}
\Scale[0.9]{\sup\limits_{(\theta) \in \Theta^{k_0-1}}\P(\a)} &\; \Scale[0.9]{\leq\; \sum\limits_{\z \in [k_0-1]^n}\sup\limits_{(\theta) \in \Theta^{k_0-1}}\P(\a, \z)}\\
& \Scale[0.9]{\leq \;(k_0-1)^n\sup\limits_{(\theta) \in \Theta^{k_0-1}}\P(\zs,\a)}
\end{split}
\end{equation*}
and on the other hand
\begin{equation*}
\begin{split}
\sup_{(\theta) \in \Theta^{k_0}}\P(\a) &\;=\; \sup_{(\theta) \in \Theta^{k_0}}\;\sum\limits_{{\bf\bar z}_n \in [k_0]^n}\P({\bf\bar z}_n,\a)\\
& \geq \; \sup_{(\theta) \in \Theta^{k_0}}\P(\z,\a)\,.
\end{split}
\end{equation*}
Therefore  
\begin{align}\label{lim1}
&\log\;\dfrac{\sup_{(\theta) \in \Theta^{k_0}}\P(\a)}{{\sup_{(\theta) \in \Theta^{k_0-1}}\P(\a)}}\notag\\
&\;\geq\; \log\; \dfrac{\sup_{(\theta) \in \Theta^{k_0}}\P(\z,\a)}{\sup_{(\theta) \in \Theta^{k_0-1}}\P(\zs,\a)} - n\log\,(k_0-1)\,.
\end{align}
Using that $n \rho_n \rightarrow \infty$  on both regimes
 $\rho_n=\rho>0$ (dense regime) and $\rho_n \rightarrow 0$ (sparse regime) we have, by 
Lemmas~\ref{lemma:ratio_underfitting_rho} and \ref{lemma:ratio_underfitting2} 
that almost surely \eqref{eq:lim_razao_proban} holds for $k=k_0-1$. 
To complete the proof, let $k < k_0-1$. In this case we can write 
\begin{align*}
&\Scale[0.9]{\log\,\dfrac{\sup_{(\theta) \in \Theta^{k_0}}\P(\a)}{{\sup_{(\theta) \in \Theta^{k}}\P(\a)}}  = \log\,\dfrac{\sup_{(\theta) \in \Theta^{k_0}}\P(\a)}{{\sup_{(\theta) \in \Theta^{k_0-1}}\P(\a)}}}\\
&\hspace{4cm} \Scale[0.9]{+ \log\,\dfrac{{\sup_{(\theta) \in \Theta^{k_0-1}}\P(\a)}}{{\sup_{(\theta) \in \Theta^{k}}\P(\a)}}\,. }
\end{align*}
The first term in the right-hand side can be handled in the same way as 
in 
\eqref{lim1}. 
On the other hand the second term  is  non-negative because the maximum likelihood function is a non-decreasing function of the dimension of the model and  $k<k_0-1$. This finishes the proof of Proposition~\ref{prop:no_underestimation}.
\end{proof}

\section{Computation and simulations}\label{simulations}

Even though the definition of $\KT{k}{\cdot}$ avoids the maximization
over the parameter space needed in other approaches as maximum likelihood, 
the computation of this function is still demanding due to the sum over the set of all possible labels. One possible way to 
approximate the integrated likelihood in  $\KT{k}{\cdot}$  is through the variational Bayes EM algorithm proposed in 
\cite{latouche2012variational}, with computational complexity of order $O(k^2n^2)$. This algorithm is implemented in the {\tt R} package \textit{mixer}.

We compare the performance of the approximate KT estimator  on simulated data with the methods  of penalized maximum 
likelihood (PML) \cite{wang2017likelihood}, Beth-Hessian matrix with moment correction (BHMC)  \cite{le2015estimating} and
 the network cross-validation method (NCV) \cite{chen2018network}. All these methods are implemented in the {\tt  R} package 
 \textit{randnet}. 

 We simulated data from a model with $k_0=3$ communities and matrix of edge probabilities
given by $P_0=\rho S_0$, with  $\rho \in \{0.02,0.04,\dots,0.3\}$. 
The matrix $S_0$ has the diagonal entries equal to $2$ and the off-diagonal entries equal to $1$. 
Figure \ref{fig:methods} shows that the approximation of the KT estimator performs well for sparse networks 
and this estimator performs better than the other methods in the case of unbalanced networks. 

This can be due to the fact that methods like the PML estimator uses spectral clustering to detect 
the communities for different values of $k$ and then choose $k$ that maximizes the penalized maximum complete likelihood. 
In general, methods based on  spectral clustering have encountered difficulties to detect the true number of communities on 
unbalanced networks, see for example  \cite{le2015estimating,wang2017likelihood}. 
On the other hand, the KT estimator takes into account the likelihood given in \eqref{eq:likelihood_function}
 without the need to carry community detection prior to the estimation of $k$. 
The mean time to compute the KT estimator for a network with $n=500$ nodes was $2.465$ minutes, while for the methods  PML, BHMC and NCV the mean time was equal to $2.10$, $0.57$ and $2.66$ seconds, respectively.


 In this simulation study we also compare the performance of the KT approach as proposed in this paper with the maximization
of the same objective function without a penalty term. The latter corresponds exactly to
the Integrated Likelihood Variational Bayes (ILvb) approach proposed in \cite{latouche2012variational}. 
Table \ref{tab:table1} shows that the  ILvb performs slightly better than the penalized KT estimator. 
However, to our knowledge there is no theoretical result  concerning the consistency of the 
non-penalized KT estimator. 

\begin{table}[h!]
\centering
\begin{tabular}{cccccccc}
\hline\hline
 & \multicolumn{7}{ c }{$\pi=(1/3,1/3,1/3)$}\\
  \hline
 $\rho$ & 0.06 & 0.08 & 0.10 & 0.12 & 0.14 & 0.16 & 0.18 \\ 
 \hline
  KT & 0.00 & 0.00 & 0.41 & 0.98 & 1.00 & 1.00 & 1.00 \\ 
  ILvb & 0.00 & 0.15 & 0.54 & 0.99 & 1.00 & 1.00 & 1.00 \\[0.5cm]
  \hline\hline
 & \multicolumn{7}{ c }{$\pi=(0.5,0.2,0.3)$}\\
   \hline 
  $\rho$ & 0.06 & 0.08 & 0.10 & 0.12 & 0.14 & 0.16 & 0.18 \\ 
  \hline
  KT & 0.00 & 0.00 & 0.04 & 0.38 & 0.80 & 0.98 & 1.00 \\ 
  ILvb & 0.05 & 0.09 & 0.26 & 0.60 & 0.86 & 0.98 & 1.00 \\
\end{tabular}
\caption{Comparison between the proportion of correct $k_0$ of the Krichevsky-Trofimov (KT) 
and the Integrated Likelihood
Variational Bayes (ILvb). We consider the model with $n=300$, $k_0=3$ and $P_0=\rho S_0$, where $S_0$ has diagonal entries equal to $2$ and off-diagonal entries equal to $1$. }
\label{tab:table1}
\end{table}

\begin{figure*}[h!]
\centering
  \begin{subfigure}[b]{0.3\textwidth}
    \includegraphics[scale=0.47]{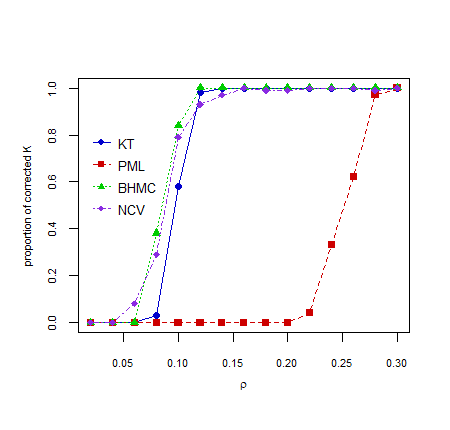}
    \caption{$n=300$ and $\pi=(1/3,1/3,1/3)$}
  \end{subfigure}
  \begin{subfigure}[b]{0.3\textwidth}
    \includegraphics[scale=0.47]{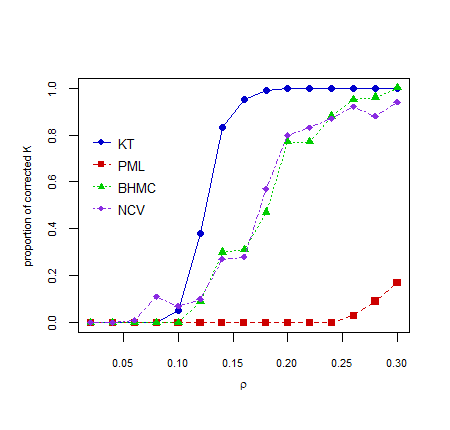}
    \caption{$n=300$ and $\pi=(0.2,0.5,0.3)$}
  \end{subfigure}
  
  \begin{subfigure}[b]{0.3\textwidth}
    \includegraphics[scale=0.47]{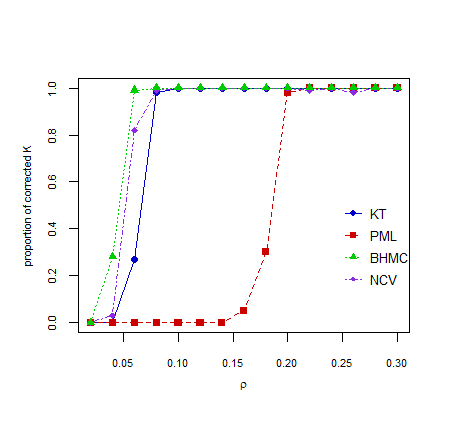}
    \caption{$n=500$ and $\pi=(1/3,1/3,1/3)$}
  \end{subfigure}
  \begin{subfigure}[b]{0.3\textwidth}
    \includegraphics[scale=0.47]{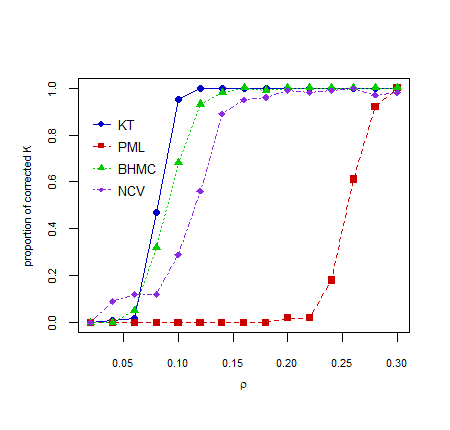}
    \caption{$n=500$ and $\pi=(0.2,0.5,0.3)$}
  \end{subfigure}
  \caption{Proportion of correct estimates for $k_0$ using the methods: Krichevsky-Trofimov (KT), Beth-Hessian matrix with moment correction (BHMC), network cross-validation (NCV) and penalized maximum likelihood (PML). We consider the model with $k_0=3$, $P_0=\rho S_0$, where $S_0$ has diagonal entries equal to $2$ and off-diagonal entries equal to $1$. 
The   tuning parameter in PML was chosen as  $\lambda=0.1$ and in KT as $\epsilon=1$. }
  \label{fig:methods}
\end{figure*}

\section{Discussion}\label{discussion}

In this paper we introduced a model selection procedure based on the Krichevsky-Trofimov mixture distribution for the number of communities in the Stochastic Block Model. We proved the almost sure convergence (strong consistency) of the penalized estimator \eqref{est_kt} to the underlying number of communities, without assuming a known upper bound on that quantity. To our knowledge this is the first strong consistency result for an estimator of the number of communities. Moreover, it is the first unbounded estimator in the sense that the set of  candidate  values is allowed to grow with the sample size. 

Our proposed penalty function  \eqref{eq:penalty} is of order $k^3\log n$, where $n$ is the number of nodes in the network, therefore for large $n$ it is considerably smaller than the $k(k+1)n\log n/2$ penalty term 
used in \cite{wang2017likelihood}. 
Even if we relate in Proposition~\ref{prop:razao_L_KT}  both objective functions,  it is not evident how to derive the consistency of the BIC as considered in \cite{wang2017likelihood} with a smaller penalty term as the one we propose here. 
 It  also remains as an open question  if the usual BIC penalty term that in this case would be $k(k+1) \log n$ leads to a consistent estimator of  the number of communities, or if the KT criterion can be consistent excluding completely the penalty term, as in the case of context tree models, see \cite{csiszar-talata-2006}. 
From the point of view of the different degree regimes, in comparison to   \cite{wang2017likelihood} we consider a wider family of sparse models with edge probability of order $\rho_n$, where $\rho_n$ can decrease to 0 at a rate such that $n\rho_n \rightarrow \infty$. 
It remains as an open question  if it is even  possible to obtain consistency  for the number of communities 
in the sparse regime with $\rho_n=1/n$.

\appendices
\section{Proofs of auxiliary results}

We begin by stating without proof a basic inequality for the Gamma function. The proof of this result can be found in \cite{davisson1981efficient}.

\begin{lemma}\label{lem:desigualdade_gamma}
For integers  $n=n_1+\cdots+n_J$ we have that
\begin{equation}
\dfrac{\prod\limits_{j=1}^J \left( \frac{n_j}{n} \right)^{n_j} }{\prod\limits_{j=1}^J  \Gamma\left(  n_j + \frac{1}{2} \right)} \;\leq \;\dfrac{1}{\Gamma\left(  n+ \frac{1}{2} \right)\Gamma\left( \frac{1}{2} \right)^{J-1}}\,.
\end{equation}
\end{lemma}

We now follow by presenting the proof a Proposition~\ref{prop:razao_L_KT}. 
\begin{proof}[Proof of Proposition~\ref{prop:razao_L_KT}]
To prove the first inequality observe that by 
definition of the KT distribution we have that 
\begin{equation*}
\begin{split}
\K{\a}&=\E_{\nu_k}[\,  \P(\a)  \,] \\
&\leq \E_{\nu_k}\Bigl[\,  \sup_{(\pi,P) \in \Theta^k} \P(\a)  \,\Bigr]\\
&= \sup_{(\pi,P) \in \Theta^k} \P(\a) \,.
\end{split}
\end{equation*}
The second inequality is based on \cite[Appendix I]{gassiat2003optimal} and \cite[Lemma~3.4]{liu1994order}. 
For $(\theta) \in \Theta^k$ we have that 
\begin{equation}\label{p_z}
\P(\z)\;=\;\prod\limits_{1\leq a\leq k} \pi_a^{n_a}
\end{equation}
and
\begin{equation}\label{p_a}
\P(\a| \z)\;=\;\prod\limits_{1\leq a \leq b\leq k}  P_{a,b}^{o_{a,b}}(1-P_{a,b})^{n_{a,b}- o_{a,b}}\,.
\end{equation}
Using that the maximum likelihood estimators for $\pi_a$ and  $P_{a,b}$ are given by $ \dfrac{n_a}{n}$ and $\dfrac{o_{a,b}}{n_{a,b}}$ respectively,  we can bound above  \eqref{p_z} and \eqref{p_a}  by 
\begin{equation}\label{p_z_estimate}
\P(\z)\;\leq\; \sup\limits_{(\theta) \in \Theta^k}\P(\z)\;=\; \prod\limits_{1\leq a\leq k}\left( \dfrac{n_a}{n}  \right)^{n_a}
\end{equation}
and 
\begin{align}\label{p_a_estimate}
\P(\a| \z)&\leq \sup\limits_{(\theta) \in \Theta^k}\P(\a| \z)\notag\\
& = \;\prod\limits_{1\leq a \leq b \leq k}  \left( \dfrac{o_{a,b}}{n_{a,b}} \right)^{o_{a,b}}\left(1-  \dfrac{o_{a,b}}{n_{a,b}}\right)^{n_{a,b}-o_{a,b}}.
\end{align}
Observe that the Krichevsky-Trofimov mixture distribution defined in \eqref{kt-mix} can be written as 
\begin{align}\label{k_a}
&\K{\a}\notag\\
&= \sum\limits_{\z \in [k]^n}\Scale[0.9]{ \biggl( \int_{\Theta^k_1} \P(\z)\nu_k^1(\pi)d\pi \biggr) \biggl(\int_{\Theta^k_2}\P(\a|\z)\nu_k^2(P)dP \biggr)}\notag \\
&= \sum\limits_{\z \in [k]^n}\K{\z} \K{\a|\z}\,,   
\end{align}
where 
\begin{equation*}
\
\nu^1_k(\pi)=  \dfrac{\Gamma(\frac{k}{2})}{\Gamma(\frac{1}{2})^k}\prod\limits_{1\leq a\leq k}\pi_{a}^{-1/2} \,,
\end{equation*}
\begin{equation*}
 \quad \nu^2_k(P)=\prod_{1\leq a \leq b \leq k}\dfrac{1}{\Gamma(\frac{1}{2})^2}P_{a,b}^{-1/2}(1-P_{a,b})^{-1/2} \,,
\end{equation*}
\[
\Theta^k_1= \{\, \pi \, | \, \pi \in (0,1]^k, \sum\limits_{a=1}^k\pi_a=1 \}\, ,
\] and 
\[
\Theta^k_2= \{  \, P  \, | \,  P \in [0,1]^{k\times k}, P  \text{ is symmetric} \, \}.
\]
We start with the evaluation of $\K{\z}$. By a simple calculation we have  that 
\begin{equation}\label{K_z}
\begin{split}
\K{\z}&\;=\; 
\dfrac{\Gamma\left(\frac{k}{2}\right)}{\Gamma\left(\frac{1}{2}\right)^k}\;\dfrac{\prod_{a=1}^k\Gamma\left(n_a+\frac{1}{2}\right)}{\Gamma\left(n+\frac{k}{2}\right)}\,.
\end{split}
\end{equation}
Then combining \eqref{p_z_estimate} and \eqref{K_z} we obtain the bound
\begin{equation}\label{eq:P_KT_z}
\dfrac{\P(\z)}{\K{\z}} \;\leq\;
\dfrac{\Gamma\left(\frac{1}{2}\right)^k \Gamma\left(n+\frac{k}{2}\right)}{\Gamma\left(\frac{k}{2}\right)}\;\prod\limits_{1\leq a\leq k}\dfrac{\left( \dfrac{n_a}{n}  \right)^{n_a}}{\Gamma\left(n_a+\frac{1}{2}\right)} \,.
\end{equation}
Using the fact that $n_1+\cdots + n_k =n$ and Lemma~\ref{lem:desigualdade_gamma} we can bound the second factor in the right-hand side of the last inequality by
\begin{equation*}\label{eq:lemma_na}
\prod\limits_{1\leq a\leq k}\dfrac{\left( \dfrac{n_a}{n}  \right)^{n_a}}{\Gamma\left(n_a+\frac{1}{2}\right)} \;\leq\; \dfrac{1}{\Gamma\left(\frac{1}{2}\right)^{k-1}\, \Gamma\left(n+\frac{1}{2}\right)}
\end{equation*}
and then we obtain the bound
\begin{equation}\label{razao_z}
\dfrac{\P(\z)}{\K{\z}} \;\leq\; \dfrac{\Gamma\left(\frac{1}{2}\right)\Gamma\left(n+\frac{k}{2}\right)}{\Gamma\left(\frac{k}{2}\right)\Gamma\left(n+\frac{1}{2}\right)}\,.
\end{equation}
The same arguments  above can be used to derive the equality
\begin{equation*}\label{K_a}
\K{\a|\z}=\prod_{1\leq a\leq b\leq k}\;\dfrac{\Gamma\left(o_{a,b}+\frac{1}{2}\right)\Gamma\left(n_{a,b} - o_{a,b}+\frac12\right)}{\Gamma\left(\frac12\right)^2\Gamma\left(n_{a,b}+1\right)}
\end{equation*}
and the 
upper bound
\begin{equation}\label{razao_a}
\begin{split}
\dfrac{\P(\a|\z)}{\K{\a|\z}} &\;\leq\;
 \prod\limits_{1\leq a \leq b \leq k}\,  \dfrac{ \Gamma\left(\textstyle\frac{1}{2}\right)\Gamma\left(n_{a,b}+1\right)}{\Gamma\left(n_{a,b}+\frac{1}{2}\right) } \,.
\end{split}
\end{equation}
Then combining the bounds in \eqref{razao_z} and \eqref{razao_a}  we  have that 
\begin{equation}\label{PKTC}
\dfrac{\P(\z,\a)}{\K{\z,\a}} \;\leq\; e^{C(\z,\a)}\,,
\end{equation}
where
\begin{align}\label{razao_final}
C(\z,\a) \;= &\; \log \Bigl( \frac{\Gamma\left(\frac{1}{2}\right)\Gamma\left(n+\frac{k}{2}\right)}{\Gamma\left(\frac{k}{2}\right)\Gamma\left(n+\frac{1}{2}\right)}\Bigr)\notag \\
& + \sum\limits_{1\leq a \leq b \leq k}\log \Bigl( \frac{\Gamma\left(\frac{1}{2}\right)\Gamma\left(n_{a,b}+1\right)}{\Gamma\left(n_{a,b}+\frac{1}{2}\right) } \Bigr)\,.
\end{align}
It can be shown, by using Stirlings' formula for the $\Gamma$
function
\[
x^{x-\frac12} e^{-x} \sqrt{2\pi} \;\leq\; \Gamma(x) \;\leq\; x^{x-\frac12}e^{-x} \sqrt{2\pi} e^{\frac1{12 x}}
\]
as in \cite[Appendix I]{gassiat2003optimal},  that 
\begin{equation}\label{log_gamma}
\begin{split}
\Scale[0.95]{\log \left( \frac{\Gamma\left(\frac{1}{2}\right)\Gamma\left(n+\frac{k}{2}\right)}{\Gamma\left(\frac{k}{2}\right)\Gamma\left(n+\frac{1}{2}\right)}\right) }& \Scale[0.95]{\;\leq\; \left(\frac{k-1}{2} \right)\log n + \frac{k(k-1)}{4n}+ \frac{1}{12n}+ \log\frac{\Gamma(\frac{1}{2})}{\Gamma(\frac{k}{2})}}\\
&\Scale[0.95]{\;\leq\; \left(\frac{k-1}{2} \right)\log n + \frac{k(k-1)}{4n}+ 1}
\end{split}
\end{equation}
and for $1\leq a \leq b\leq k$,   
\begin{equation}\label{log_gamma2}
\begin{split}
\Scale[0.95]{\log \left( \frac{\Gamma\left(\frac{1}{2}\right)\Gamma\left(n_{a,b}+1\right)}{\Gamma\left(n_{a,b}+\frac{1}{2}\right) } \right)} &\; \Scale[0.95]{\leq\; \frac{1}{2}\log n_{a,b} + \frac{1}{2n_{a,b}} + \frac{1}{12 n_{a,b}} + \log \Gamma(\frac12)} \\
&\Scale[0.95]{\; \leq\; \log n +  \frac{3}{2}}
\end{split}
\end{equation}
where the last inequality follows from the fact that $n_{a,b}\leq n^2$  for all $a$ and $b$.
Summing the bounds \eqref{log_gamma} and \eqref{log_gamma2}  we obtain  in \eqref{razao_final} that 
\begin{equation}\label{cza}
C(\z,\a)  \;\leq \;  \textstyle\frac{k(k+2)-1}{2}   \log n + c_k\,,
\end{equation}
where  
\[
c_{k} =k(k+1) + 1\,.
\]
Using the uniform bound \eqref{cza} and \eqref{PKTC} we can finally write
\begin{equation*}
\begin{split}
&\Scale[0.95]{\log\sup\limits_{(\pi,P) \in \Theta^k} \P(\a)\,\;}  \\
&\Scale[0.95]{\leq  \sup\limits_{(\pi,P) \in \Theta^k}\log\biggl(\; \sum\limits_{\z \in [k]^n}e^{C(\z,\a)}\K{\z}\K{\a|\z}\biggr)}\\
&\leq \log\K{\a} +  \textstyle\frac{k(k+2)-1}{2}  \log n + c_{k}
\end{split}
\end{equation*}
and this concludes the proof. 
\end{proof}


Now we state and prove a lemma that is useful to bound
the probability of overestimation.

\begin{lemma}\label{prop:ineq_prob_k} 
For $k > k_0$
we have 
\begin{equation*}
\begin{split}
&\Po( \hat{k}_{\mbox{\tiny{KT}}}(\a) = k )\\
&\hspace{2cm} \; \leq\; \exp\left\lbrace   \textstyle\frac{k_0(k_0+2)-1}{2}\log n + c_{k_0} + d_{k_0,k,n} \right\rbrace\,,
\end{split}
\end{equation*}
where  $d_{k_0,k,n} = \pen{k_0,n}-\pen{k,n}$.
\end{lemma}
\begin{proof}
For $k > k_0$ define the events
\[
A(k) \;=\; \Bigl\{\argmax\limits_{1\leq k'\leq n} \{ \,  \log \KT{k'}{\a} - \frac12 (k')^3\log n \,\} = k\Bigr\}
\]
and
\[
B(k) \;=\; \Bigl\{\KT{k_0}{\a} \leq \KT{k}{\a} \text{e}^{d_{k_0,k,n}}\Bigr\}\,,
\]
where 
\[
d_{k_0,k,n} =  \pen{k_0,n}-\pen{k,n}\,.
\]
Observe that $A(k) \subset B(k)$ for all $k$. 
Then 
\begin{align}\label{over_1}
\Po( \hat{k}_{\mbox{\tiny{KT}}}(\a) = k )  \;&= \;\sum\limits_{\a}\Po(\a)\mathds{1}_{A(k)}
\notag \\
& \leq \; \sum\limits_{\a}\Po(\a)\mathds{1}_{B(k)}\,.
\end{align}
By Proposition~\ref{prop:razao_L_KT}
\begin{align*}
\Scale[0.95]{\log \Po(\a)} &\;\leq\; \Scale[0.95]{\log\sup\limits_{(\theta) \in \Theta^{k_0}} \P(\a)}\\
&\; \Scale[0.95]{\leq \;\log\KT{k_0}{\a}+\frac{k_0(k_0
+2)-1}{2} \log n + c_{k_0}}
\end{align*}
and therefore 
\begin{equation}\label{over_2}
\Po(\a) \;\leq\;  \KT{k_0}{\a}n^{\frac{k_0(k_0
+2)-1}{2}}e^{ c_{k_0}}\,.
\end{equation}
Applying \eqref{over_2} in \eqref{over_1} we obtain that 
\begin{align*}
\Po( \hat{k}_{\mbox{\tiny{KT}}}&(\a) = k )\\
& \;\leq\; \sum\limits_{\a}\KT{k_0}{\a}n^{\frac{k_0(k_0
+2)-1}{2}}e^{ c_{k_0}}\mathds{1}_{B(k)}\\
&  \;\leq\; \sum\limits_{\a}\KT{k}{\a}\, e^{d_{k_0,k,n}} n^{\frac{k_0(k_0+2)-1}{2}}e^{ c_{k_0}}\\
& = \exp\left\lbrace  \textstyle\frac{k_0(k_0+2)-1}{2}\log n + c_{k_0} + d_{k_0,k,n} \right\rbrace\,,
\end{align*}
where the last equality follows from the fact that $\K{\cdot}$ is a probability distribution on the space of adjacency matrices.
This concludes the proof of Lemma~\ref{prop:ineq_prob_k}.
\end{proof}

\section*{Acknowledgment}
This work was produced as part of the activities of FAPESP Research, 
Innovation and Dissemination Center for Neuromathematics, 
grant 2013/07699-0, and FAPESP's project  
\emph{Model selection in high dimensions: theoretical properties and applications}, grant 2019/17734- 3, S\~ao Paulo Research Foundation, Brazil. 
AC was supported by grant 2015/12595-4, S\~ao 
Paulo Research Foundation (FAPESP). 
Part of this work was completed while AC was a Visiting Postdoctoral Researcher at University of Michigan.
She thanks the support and hospitality of this institution.

\ifCLASSOPTIONcaptionsoff
  \newpage
\fi



\bibliographystyle{IEEEtran}
\bibliography{references}

\begin{IEEEbiography}
[{\includegraphics[width=1in,height=1.25in,clip,keepaspectratio]{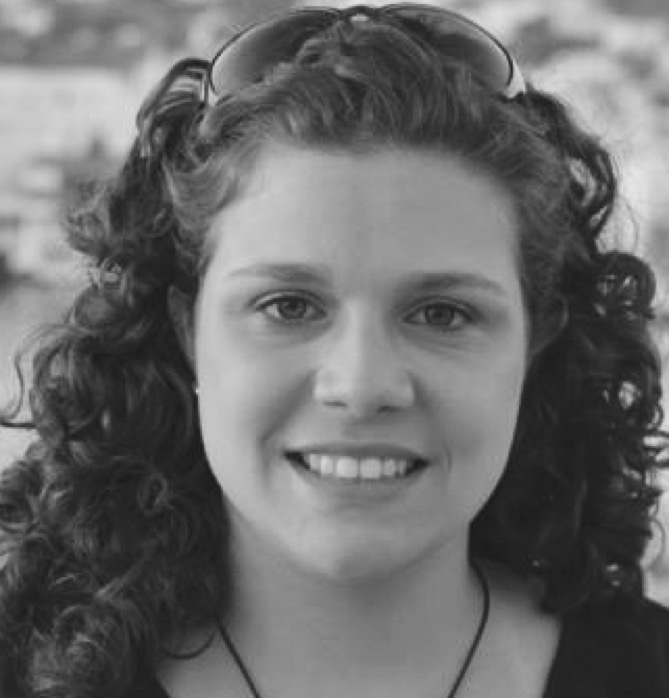}}]
{Andressa Cerqueira} 
received a PhD degree in Statistics from the University of S\~ao Paulo in 2018. 
During 2018 she was a postdoctoral researcher at University of Campinas. 
During 2019 she spent one year as a visiting postdoctoral researcher at University of Michigan. Since 2020 she is Assistant Professor at Department of Statistics at
Federal University of S\~ao Carlos (UFSCar) and her research is focused on inference in networks.
\end{IEEEbiography}

\vspace*{-5mm}

\begin{IEEEbiography}[{\includegraphics[width=1in,height=1.25in,clip,keepaspectratio]{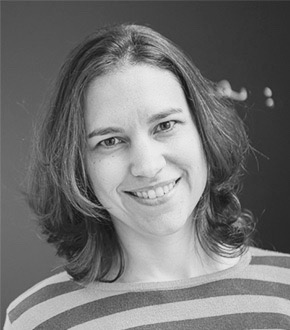}}]{Florencia Leonardi}
received a bachelor's degree in Mathematics from the University of Mar del Plata in 2002 and a PhD degree in 
Bioinformatics from the University of S\~ao Paulo in 2007.  
She became  Adjoint Professor in 2008 at the Institute of Mathematics and Statistics of the University of S\~ao Paulo, and Associate Professor since 2017. 
During 2014- 2015 she spent a sabbatical year as a Visiting Professor at ETHZ in Switzerland. 
Her main research interests are inference and applications of stochastic processes.
\end{IEEEbiography}

\vfill

\end{document}